\documentclass{article}
\usepackage{amssymb}
\usepackage{amsmath}
\usepackage{amsfonts}

\newtheorem{theorem}{Theorem}

\newtheorem{corollary}[theorem]{Corollary}

\newtheorem{proposition}[theorem]{Proposition}

\newenvironment{proof}[1][Proof]{\noindent\textbf{#1.} }{\ \rule{0.5em}{0.5em}}
\input{tcilatex}
\begin{document}

\title{Inequalities and Asymptotics for some Moment Integrals }
\author{Faruk Abi-Khuzam\thanks{%
2010 \textit{Mathematics Subject Classification. Primary }52A20}\thanks{%
\textit{Key words and phrases}. Ball's inequality,asymptotics,cube slicing,
moments.} \\
%EndAName
Department of Mathematics\\
American University of Beirut\\
Beirut, Lebanon}
\maketitle

\begin{abstract}
For $\alpha >\beta -1>0$, we obtain two sided inequalities for the moment
integral $I(\alpha ,\beta )=$ $\int_{%
%TCIMACRO{\U{211d} }%
%BeginExpansion
\mathbb{R}
%EndExpansion
}|x|^{-\beta }|\sin x|^{\alpha }dx$ .These are then used to give the exact
asymptotic behavior of the integral as $\alpha \rightarrow \infty $. The
case $I(\alpha ,\alpha )$ corresponds to the asymptotics of Ball's
inequality, and $I(\alpha ,[\alpha ]-1)$ corresponds to a kind of novel
"oscillatory" behavior.
\end{abstract}

\section{\protect\bigskip Introduction}

Ball's integral inequality \cite{KB} , in connection with cube-slicing in $%
%TCIMACRO{\U{211d} }%
%BeginExpansion
\mathbb{R}
%EndExpansion
^{n}$, says that for all $s\geq 2,$%
\begin{equation*}
\int_{-\infty }^{\infty }|\frac{\sin (\pi x)}{\pi x}|^{s}dx\leq \sqrt{\frac{2%
}{s}},\text{ or }\int_{-\infty }^{\infty }|\frac{\sin x}{x}|^{s}dx\leq \pi 
\sqrt{\frac{2}{s}},
\end{equation*}%
with strict inequality except when $s=2$. In particular, it suggests that
the integral decays like $\frac{1}{\sqrt{s}}$ as $s\rightarrow \infty $, and
this is made precise by the following asymptotic \cite{NP} :%
\begin{equation*}
\lim_{s\rightarrow \infty }\sqrt{\frac{s}{2}}\int_{-\infty }^{\infty }|\frac{%
\sin (\pi x)}{\pi x}|^{s}dx=\sqrt{\frac{3}{\pi }}.
\end{equation*}%
Since $\sqrt{\frac{3}{\pi }}<1$, the asymptotic result implies the
inequality for large values of $s$. But there are no known "easy" proofs of
the inequality for the full range of values,the main difficulty being near
small values of $s$. e.g. between$\ 2$ and $4$.\cite{NP}.$\ $The asymptotic
result, though reasonably tame, presents new difficulties when we consider a
more general integral, and this is circumvented here by the proof of two new
inequalities.

Our purpose here is to consider a generalization involving the " moment"
integral 
\begin{equation*}
I(\alpha ,\beta )=\int_{-\infty }^{\infty }\frac{|\sin x|^{\alpha }}{%
\left\vert x\right\vert ^{\beta }}dx,\alpha >\beta -1>0.
\end{equation*}%
We shall obtain useful upper and lower bounds for this integral, and use
them to obtain the asymptotic behavior of this integral. In addition, the
inequalities obtained are indispensible in obtaining the asymptotic
behavior, especially in the interesting "oscillatory" cases, $I(\alpha
,[\alpha ])$ if $\alpha \geq 2$, and $I([\alpha ],\alpha -1)$ if $\alpha >1$%
, where $[\alpha ]$ is the greatest integer in $\alpha $. The oscillatory
behavior makes it impossible to employ the standard methods used in
connection with Ball's inequality.

We place no restrictions on the indices $\alpha $ and $\beta $ beyond those
necessary to ensure the convergence of the integral $I(\alpha ,\beta )$.
Indeed, the condition $\beta >1$ implies convergence in a nbhd of $\infty
\allowbreak $, and near $0$, the inequality $\frac{|\sin x|^{\alpha }}{%
\left\vert x\right\vert ^{\beta }}\leq |x|^{\alpha -\beta }$ implies
convergence, since $\alpha -\beta >-1$.

\section{\protect\bigskip Weaker versions of Ball's Inequality}

A natural way to deal with Ball's inequality is to apply the sharp form of
the Hausdorff-Young inequality \cite{WB}. This leads to two inequalities for
the relevant integral: the first works for all $s\geq 2$, but falls short of
the required inequality by supplying the larger constant $\sqrt{e}$ in place
of $\sqrt{2}.$ The second gives a constant smaller than $\sqrt{2}$ but only
works for $s\geq 4.$

\begin{proposition}
(a) If $s\geq 2$,then 
\begin{equation*}
\int_{-\infty }^{\infty }|\frac{\sin (\pi x)}{\pi x}|^{s}dx\leq \frac{1}{%
\sqrt{s}}\cdot \sqrt{\left( 1+\frac{1}{s-1}\right) ^{s-1}}<\sqrt{\frac{e}{s}}%
.
\end{equation*}%
(b) If $s\geq 4,q=\frac{s}{2}$, and $p$ is the index conjugate to $q,$ then%
\begin{equation*}
\int_{%
%TCIMACRO{\U{211d} }%
%BeginExpansion
\mathbb{R}
%EndExpansion
}\left\vert \left( \frac{\sin \pi \xi }{\pi \xi }\right) \right\vert
^{s}\leq \sqrt{\frac{2}{s}}\left( \frac{2\sqrt{p}}{p+1}\right) ^{q/p}<\sqrt{%
\frac{2}{s}}.
\end{equation*}
\end{proposition}

\begin{proof}
\textit{For part (a),} \textit{let }$\chi =\chi _{(-1/2,1/2)}$\textit{\ be
the characteristic function of the interval }$(-1/2,1/2)$\textit{. Then its
Fourier transform is given by }$\hat{\chi}(\xi )=\frac{\sin \pi \xi }{\pi
\xi }.$\textit{\ Applying the sharp Hausdorff-Young inequality \cite{WB}, }$%
\Vert \hat{\chi}\Vert _{s}\leq C_{p}\Vert \chi \Vert _{p}$\textit{, where }$%
s\geq 2,$\textit{\ }$p=s^{\prime }$\textit{, the index conjugate to }$s$%
\textit{, and }$C_{p}$\textit{\ is given by }$C_{p}^{2}=p^{1/p}(s)^{-1/s},$%
\textit{\ we obtain}%
\begin{equation*}
\int_{R}|\frac{\sin \pi \xi }{\pi \xi }|^{s}d\xi \leq \left(
p^{1/p}s^{-1/s}\right) ^{s/2}.
\end{equation*}%
\textit{It now remains to compute}%
\begin{equation*}
\left( p^{1/p}s^{-1/s}\right) ^{s/2}=\frac{1}{\sqrt{s}}\cdot \left( \frac{s}{%
s-1}\right) ^{\frac{s-1}{2}}=\frac{1}{\sqrt{s}}\cdot \sqrt{\left( 1+\frac{1}{%
s-1}\right) ^{s-1}}<\sqrt{\frac{e}{s}},s\geq 2,
\end{equation*}%
\textit{and the inequality in (a) follows.}
\end{proof}

\textit{To prove part (b), we employ the convolution }$g=\chi \ast \chi $%
\textit{\ of the same characteristic function. A simple computation gives }%
\begin{equation*}
g(x)=\left\{ 
\begin{array}{cc}
1-x & 0\leq x\leq 1 \\ 
1+x & -1\leq x\leq 0 \\ 
0 & |x|\geq 1%
\end{array}%
\right\} .
\end{equation*}%
\textit{Now }$\hat{g}(\xi )=\left( \frac{\sin \pi \xi }{\pi \xi }\right)
^{2} $\textit{, }$\Vert g\Vert _{p}^{p}=\frac{2}{p+1}$\textit{, and an
application of the sharp-Hausdorff-Young inequality gives, for }$q\geq 2$%
\textit{, and the conjugate index }$p=q^{\prime }$\textit{, }%
\begin{equation*}
\int_{%
%TCIMACRO{\U{211d} }%
%BeginExpansion
\mathbb{R}
%EndExpansion
}\left\vert \left( \frac{\sin \pi \xi }{\pi \xi }\right) ^{2}\right\vert
^{q}\leq \left( p^{1/p}q^{-1/q}\right) ^{q/2}\left( \frac{2}{p+1}\right)
^{q/p}=q^{-1/2}\left( \frac{2\sqrt{p}}{p+1}\right) ^{q/p}<\frac{2\sqrt{p}}{%
p+1}\cdot \sqrt{\frac{1}{q}}.
\end{equation*}%
\textit{Since }$\frac{2\sqrt{p}}{p+1}<1,$ we \textit{obtain the inequality, }%
$\int_{-\infty }^{\infty }|\frac{\sin (\pi x)}{\pi x}|^{s}dx<\sqrt{\frac{2}{s%
}}$\textit{\ for all }$s\geq 4$\textit{.}

\section{Main Results}

In this section we consider the question of obtaining upper and lower bounds
for the more general integral, namely $\int_{-\infty }^{\infty }\frac{|\sin
x|^{\alpha }}{\left\vert x\right\vert ^{\beta }}dx$. Those bounds are then
used to obtain the precise asymptotic behaviour of the integral as $\alpha
\rightarrow \infty $. In addition, the bounds make it possible to employ
discontinuous functions such as $[\alpha ]$ in place of $\beta $, and then
the asymptotic result also captures the precise oscillations in the values
of the integral, as $\alpha \rightarrow \infty $.

\begin{theorem}
Suppose $\alpha >\beta -1>0$, and put 
\begin{equation*}
I(\alpha ,\beta )=\int_{%
%TCIMACRO{\U{211d} }%
%BeginExpansion
\mathbb{R}
%EndExpansion
}\frac{|\sin t|^{\alpha }}{|t|^{\beta }}dt=2\int_{0}^{\infty }\frac{|\sin
t|^{\alpha }}{|t|^{\beta }}dt,
\end{equation*}%
and 
\begin{equation*}
\phi (\alpha ,\beta )=\frac{\alpha ^{\frac{\alpha -\beta +1}{2}}\Gamma
(\alpha +1)}{\Gamma (\frac{\alpha -\beta +1}{2}+\alpha +1)},
\end{equation*}%
where $\Gamma $ is the gamma-function.Then
\end{theorem}

\begin{equation*}
\left( \frac{6}{\alpha }\right) ^{\frac{\alpha -\beta +1}{2}}\Gamma (\frac{%
\alpha -\beta +1}{2})\phi (\alpha ,\beta )\leq I(\alpha ,\beta )\leq \left( 
\frac{6}{\alpha }\right) ^{\frac{\alpha -\beta +1}{2}}\Gamma \left( \frac{%
\alpha -\beta +1}{2}\right) \left\{ 1+\frac{1}{\beta -1}\right\} .
\end{equation*}%
\textit{In particular, if }$\beta =\alpha $, then%
\begin{equation*}
\sqrt{\frac{6}{\alpha }}\frac{\sqrt{\alpha }\Gamma (\alpha +1)}{\Gamma
(\alpha +\frac{3}{2})}\sqrt{\pi }\leq I(\alpha ,\alpha )\leq \left( \sqrt{%
\frac{6}{\alpha }}\right) \sqrt{\pi }\left\{ 1+\frac{1}{\alpha -1}\right\} .
\end{equation*}

\begin{proof}
\textit{We need first the following double inequality,}
\end{proof}

\begin{equation*}
1-\frac{x^{2}}{6}\leq \frac{\sin x}{x}\leq e^{-\frac{x^{2}}{6}},0\leq x\leq
\pi .
\end{equation*}%
\textit{The left-hand inequality is easily proved by calculus. It will be
used with }$0\leq x\leq \sqrt{6}$\textit{\ }$.$\textit{For the right-hand
inequality, since }$0\leq x\leq \pi $\textit{, we may use the inequality
between the geometric and arithmetic mean of positive numbers to obtain} 
\begin{equation*}
\dprod\limits_{k=1}^{n}\left( 1-\frac{x^{2}}{\pi ^{2}k^{2}}\right) \leq
\left( 1-\frac{x^{2}}{\pi ^{2}n}\sum_{k=1}^{n}\frac{1}{k^{2}}\right)
^{n}\leq \exp \left( -\frac{x^{2}}{\pi ^{2}}\sum_{k=1}^{n}\frac{1}{k^{2}}%
\right) .
\end{equation*}%
\textit{Letting }$n\rightarrow \infty $\textit{, and recalling the product
representation of the sine function, and }$\sum_{k=1}^{\infty }\frac{1}{k^{2}%
}=\frac{\pi ^{2}}{6}$\textit{, we obtain the second inequality. The next
step is to compare the full integral in the theorem to an integral over the
interval }$[0,\sqrt{6}]$, \textit{or over} $[0,\pi ].$\textit{\ }%
\begin{equation*}
\int_{0}^{\sqrt{6}}\frac{|\sin t|^{\alpha }}{|t|^{\beta }}dt\leq
\int_{0}^{\pi }\frac{|\sin t|^{\alpha }}{|t|^{\beta }}dt\leq
\int_{0}^{\infty }\frac{|\sin t|^{\alpha }}{|t|^{\beta }}dt
\end{equation*}%
\begin{equation*}
=\int_{0}^{\pi }\frac{|\sin t|^{\alpha }}{|t|^{\beta }}dt+\sum_{k=1}^{\infty
}\int_{k\pi }^{(k+1)\pi }\frac{|\sin t|^{\alpha }}{|t|^{\beta }}dt
\end{equation*}%
\begin{equation*}
=\int_{0}^{\pi }\frac{|\sin t|^{\alpha }}{|t|^{\beta }}dt+\sum_{k=1}^{\infty
}\int_{0}^{\pi }\frac{|\sin t|^{\alpha }}{|t+k\pi |^{\beta }}dt
\end{equation*}%
\begin{equation*}
\leq \int_{0}^{\pi }\frac{|\sin t|^{\alpha }}{|t|^{\beta }}dt\left\{
1+\sum_{k=1}^{\infty }\frac{1}{(k+1)^{\beta }}\right\}
\end{equation*}%
\begin{equation*}
\leq \left\{ 1+\frac{1}{\beta -1}\right\} \int_{0}^{\pi }\frac{\sin ^{\alpha
}t}{t^{\beta }}dt.
\end{equation*}%
\textit{Using} \textit{the above inequalities for }$\frac{\sin x}{x}$\textit{%
, }%
\begin{equation*}
\int_{0}^{\sqrt{6}}t^{\alpha -\beta }\left( 1-\frac{t^{2}}{6}\right)
^{\alpha }dt\leq \int_{0}^{\pi }\frac{\sin ^{\alpha }t}{t^{\beta }}dt\leq
\int_{0}^{\pi }t^{\alpha -\beta }\exp \left( -\frac{\alpha t^{2}}{6}\right)
dt.
\end{equation*}%
\textit{Simple substitutions to change variables bring this double
inequality to the form}%
\begin{equation*}
\frac{1}{2}6^{\frac{\alpha -\beta +1}{2}}\int_{0}^{1}x^{\frac{\alpha -\beta
-1}{2}}(1-x)^{\alpha }dx\leq \int_{0}^{\pi }\frac{\sin ^{\alpha }t}{t^{\beta
}}dt\leq \frac{1}{2}\left( \frac{6}{\alpha }\right) ^{\frac{\alpha -\beta +1%
}{2}}\int_{0}^{\frac{\pi ^{2}\alpha }{6}}x^{\frac{\alpha -\beta -1}{2}%
}e^{-x}dx.
\end{equation*}%
\textit{If we extend the right most integral to }$[0,\infty )$\textit{, and
then express both sides through the gamma function, we arrive at }%
\begin{equation*}
\frac{1}{2}6^{\frac{\alpha -\beta +1}{2}}\frac{\Gamma (\frac{\alpha -\beta +1%
}{2})\Gamma (\alpha +1)}{\Gamma (\frac{\alpha -\beta +1}{2}+\alpha +1)}\leq
\int_{0}^{\pi }\frac{\sin ^{\alpha }t}{t^{\beta }}dt\leq \frac{1}{2}\left( 
\frac{6}{\alpha }\right) ^{\frac{\alpha -\beta +1}{2}}\Gamma (\frac{\alpha
-\beta +1}{2}).
\end{equation*}%
\textit{This gives the first inequalities for }$I(\alpha ,\beta )$\textit{,
and so, the inequalities for }$I(\alpha ,\alpha )$\textit{.}

\begin{corollary}
Let $I(\alpha ,\beta )$ be the integral in the theorem.
\end{corollary}

(a)\textit{\ }I\textit{f }$\alpha -\beta =c$\textit{\ is held constant,
while }$\alpha \rightarrow \infty $\textit{, then }%
\begin{equation*}
\lim_{\alpha \rightarrow \infty }\alpha ^{\frac{c+1}{2}}I(\alpha ,\beta )=6^{%
\frac{c+1}{2}}\Gamma (\frac{c+1}{2}),c>-1.
\end{equation*}%
\textit{In particular, the asymptotic for the integral in Ball's inequality
is} 
\begin{equation*}
\lim_{\alpha \rightarrow \infty }\sqrt{\alpha }I(\alpha ,\alpha )=\sqrt{6\pi 
}.
\end{equation*}

(b) \textit{If }$\alpha -\beta =c$, and $c$ \textit{\ remains bounded as }$%
\alpha \rightarrow \infty ,$\textit{\ then }%
\begin{equation*}
I(\alpha ,\beta )\backsim \left( \frac{6}{\alpha }\right) ^{\frac{\alpha
-\beta +1}{2}}\Gamma \left( \frac{\alpha -\beta +1}{2}\right) ,\alpha
\rightarrow \infty .
\end{equation*}%
\textit{In particular,}%
\begin{equation*}
I(\alpha ,[\alpha ])\backsim \left( \frac{6}{\alpha }\right) ^{\frac{\alpha
-[\alpha ]+1}{2}}\Gamma \left( \frac{\alpha -[\alpha ]+1}{2}\right) ,\alpha
\rightarrow \infty .
\end{equation*}

\begin{proof}
\textit{(a) In the very special case where }$\beta =\alpha $,\textit{\
Stirling's formula gives }%
\begin{equation*}
\phi (\alpha ,\alpha )=\frac{\sqrt{\alpha }\Gamma (\alpha +1)}{\Gamma
(\alpha +\frac{3}{2})}\thicksim \frac{\alpha ^{\alpha +1}e^{1/2}}{(\alpha
+1/2)^{\alpha +1}}=\frac{e^{1/2}}{(1+\frac{1}{2\alpha })^{\alpha +1}}%
\thicksim 1,\alpha \rightarrow \infty .
\end{equation*}%
\textit{From this, the case where }$\alpha -\beta =c$\textit{, a constant,
is handled similarly :}%
\begin{equation*}
\phi (\alpha ,\beta )=\frac{\alpha ^{\frac{c+1}{2}}\Gamma (\alpha +1)}{%
\Gamma (\frac{c+1}{2}+\alpha +1)}\thicksim \frac{\alpha ^{^{\alpha +\frac{c+1%
}{2}+\frac{1}{2}}}e^{-\alpha }}{(\alpha +\frac{c+1}{2})^{\alpha +\frac{c+1}{2%
}+\frac{1}{2}}e^{-(\alpha +\frac{c+1}{2})}}
\end{equation*}%
\begin{equation*}
=\frac{e^{\frac{c+1}{2}}}{(1+\frac{c+1}{2\alpha })^{\alpha +\frac{c+1}{2}+%
\frac{1}{2}}}\thicksim 1,\alpha \rightarrow \infty .
\end{equation*}
\end{proof}

(c) \textit{If }$\alpha -\beta =c>-1$, and\textit{\ }$c$ \textit{is only
bounded, then Stirling's formula followed by the inequality }$(1+\frac{c+1}{%
2\alpha })^{\alpha }\leq e^{\frac{c+1}{2}}$, gives\textit{\ }%
\begin{equation*}
\liminf_{\alpha \rightarrow \infty }\phi (\alpha ,\beta )=\liminf_{\alpha
\rightarrow \infty }\frac{\alpha ^{\frac{c+1}{2}}\Gamma (\alpha +1)}{\Gamma (%
\frac{c+1}{2}+\alpha +1)}
\end{equation*}%
\begin{equation*}
=\liminf_{\alpha \rightarrow \infty }\frac{e^{\frac{c+1}{2}}}{(1+\frac{c+1}{%
2\alpha })^{\alpha +\frac{c+1}{2}+\frac{1}{2}}}\geq \liminf_{\alpha
\rightarrow \infty }\frac{1}{(1+\frac{c+1}{2\alpha })^{\frac{c+1}{2}+\frac{1%
}{2}}}=1.
\end{equation*}%
\textit{So that }%
\begin{equation*}
\liminf_{\alpha \rightarrow \infty }\left[ \left( \frac{6}{\alpha }\right) ^{%
\frac{\alpha -\beta +1}{2}}\Gamma (\frac{\alpha -\beta +1}{2})\right]
^{-1}I(\alpha ,\beta )\geq 1.
\end{equation*}%
\textit{The corresponding }$\limsup $\textit{\ being clearly }$\leq 1$%
\textit{, we obtain }%
\begin{equation*}
I(\alpha ,\beta )\backsim \left( \frac{6}{\alpha }\right) ^{\frac{\alpha
-\beta +1}{2}}\Gamma \left( \frac{\alpha -\beta +1}{2}\right) ,\alpha -\beta 
\text{ \textit{bounded},}\beta \rightarrow \infty .
\end{equation*}

\section{Conclusion}

We conclude by a generalization of the asymptotic result for a class of
infinite products. Let $g$ be a function having an infinite product
representation of the form

\begin{equation*}
g(t)=\dprod\limits_{n=1}^{\infty }\left( 1-\frac{t^{2}}{t_{n}^{2}}\right) ,
\end{equation*}%
where $t_{n}>0$, and $c=\sum_{n=1}^{\infty }t_{n}^{-2}<\infty .$ We are
interested in investigating $\lim_{p\rightarrow \infty }\int_{0}^{\infty }%
\frac{|g(t)|^{p}}{t^{\beta }}dt$, where $0\leq \beta <1$. Two examples of
such a function are: 
\begin{equation*}
f(x)=\int_{0}^{1}\cos (xt)h(t)dt,J_{0}(x)=\frac{2}{\pi }\int_{0}^{1}\frac{%
\cos xt}{\sqrt{1-t^{2}}}dt.
\end{equation*}%
The first function $f$ $\ $was considered in \cite{KK} in connection with
maximal measures of sections of the $n$-cube. The second is the Bessel
function of order $0$.

We first review the case where $\beta =0$. If $0\leq t\leq t_{1}$, we need
two inequalities analogous to those obtained for the sinc function. If $%
0<a_{i}<1$, then we use the double inequality%
\begin{equation*}
1-(a_{1}+a_{2}+...+a_{n})\leq \dprod\limits_{k=1}^{n}\left( 1-a_{k}\right)
\leq \left( 1-\frac{1}{n}\sum_{k=1}^{n}a_{k}\right) ^{n},
\end{equation*}%
to obtain 
\begin{equation*}
\left( 1-t^{2}\sum_{k=1}^{n}t_{k}^{-2}\right) \leq
\dprod\limits_{k=1}^{n}\left( 1-\frac{t^{2}}{t_{k}^{2}}\right) \leq \left( 1-%
\frac{t^{2}}{n}\sum_{k=1}^{n}t_{k}^{-2}\right) ^{n}\leq \exp \left(
-t^{2}\sum_{k=1}^{n}t_{k}^{-2}\right) ,
\end{equation*}%
and letting $n\rightarrow \infty $, get 
\begin{equation*}
1-ct^{2}\leq |g(t)|\leq e^{-ct^{2}},c=\sum_{k=1}^{\infty }t_{k}^{-2},
\end{equation*}%
where the left inequality is used when $0\leq t\leq \frac{1}{\sqrt{c}}$, and
the right inequality when $0\leq t\leq t_{1}$.

The left-hand inequality gives 
\begin{equation*}
\int_{0}^{\infty }|g(t)|^{p}dt\geq \int_{0}^{\frac{1}{\sqrt{c}}%
}(1-ct^{2})^{p}dt=\frac{1}{2\sqrt{c}}\int_{0}^{1}(1-x)^{p}x^{-1/2}dx=\frac{1%
}{2\sqrt{c}}\frac{\Gamma (p+1)\sqrt{\pi }}{\Gamma (p+\frac{3}{2})}.
\end{equation*}%
By Stirling's formula we obtain 
\begin{equation*}
\liminf_{p\rightarrow \infty }\sqrt{p}\int_{-\infty }^{\infty
}|g(t)|^{p}dt\geq \sqrt{\frac{\pi }{c}},
\end{equation*}%
which suggests that the order of decay of the integral is $\frac{1}{\sqrt{p}}
$, and so leads naturally a consideration of $\sqrt{p}\int_{0}^{\infty
}|g(t)|^{p}dt=\int_{0}^{\infty }|g(\frac{t}{\sqrt{p}})|^{p}dt$. Now use of
two sided inequalities gives 
\begin{equation*}
(1-c\frac{t^{2}}{p})^{p}\leq |g(\frac{t}{\sqrt{p}})|^{p}\leq e^{-ct^{2}},
\end{equation*}%
where the left-hand inequality holds true for $0\leq t\leq \frac{\sqrt{p}}{%
\sqrt{c}},$ and the rigt-hand inequality holds true for $0\leq t\leq t_{1}%
\sqrt{p}$. It now becomes possible to use, exactly as done in \cite{NP},
Lebesgue's dominated convergence theorem to conclude that actually $%
\lim_{p\rightarrow \infty }\sqrt{p}\int_{-\infty }^{\infty }|g(t)|^{p}dt=%
\sqrt{\frac{\pi }{c}}$.

In the general case where $0<\beta <1$, if we were to try the same approach,
we would need to know beforehand the expected rate of decay. Thus using one
of the inequalities above, we obtain 
\begin{equation*}
\int_{0}^{\infty }\frac{|g(t)|^{p}}{t^{\beta }}dt\geq \int_{0}^{\frac{1}{%
\sqrt{c}}}(1-ct^{2})^{p}t^{-\beta }dt=\frac{1}{2\sqrt{c^{1-\beta }}}%
\int_{0}^{1}(1-x)^{p}x^{-\frac{1+\beta }{2}}dx
\end{equation*}%
\begin{equation*}
=\frac{1}{2\sqrt{c^{1-\beta }}}\frac{\Gamma (p+1)\Gamma (\frac{1-\beta }{2})%
}{\Gamma (\frac{1-\beta }{2}+p+1)},
\end{equation*}%
leading to a sharp lower asymptotic, namely 
\begin{equation*}
\liminf_{p\rightarrow \infty }p^{\frac{1-\beta }{2}}\int_{0}^{\infty }\frac{%
|g(t)|^{p}}{t^{\beta }}dt\geq \frac{\Gamma (\frac{1-\beta }{2})}{2\sqrt{%
c^{1-\beta }}}\text{.}
\end{equation*}%
Once again this suggests that the expected decay is like $p^{\frac{\beta -1}{%
2}}$. So we make the substitution \ $t=(1-\beta )^{\frac{1}{1-\beta }%
}p^{-1/2}x^{\frac{1}{1-\beta }\text{ }},$ and find that 
\begin{equation*}
p^{\frac{1-\beta }{2}}\int_{0}^{\infty }\frac{|g(t)|^{p}}{t^{\beta }}%
=\int_{0}^{\infty }\left\vert g\left( \frac{(1-\beta )^{\frac{1}{1-\beta }%
}x^{\frac{1}{1-\beta }\text{ }}}{\sqrt{p}}\right) \right\vert ^{p}dx.
\end{equation*}%
The inequalities 
\begin{equation*}
\left( 1-c\frac{(1-\beta )^{\frac{2}{1-\beta }}x^{\frac{2}{1-\beta }\text{ }}%
}{p}\right) ^{p}\leq \left\vert g\left( \frac{(1-\beta )^{\frac{1}{1-\beta }%
}x^{\frac{1}{1-\beta }\text{ }}}{\sqrt{p}}\right) \right\vert ^{p}\leq \exp
\left( -c(1-\beta )^{\frac{2}{1-\beta }}x^{\frac{2}{1-\beta }\text{ }}\right)
\end{equation*}%
make it possible to use Lebesgue's theorem and we arrive at the asymptotic%
\begin{equation*}
\lim_{p\rightarrow \infty }p^{\frac{1-\beta }{2}}\int_{0}^{\infty }\frac{%
|g(t)|^{p}}{t^{\beta }}dt=\frac{\Gamma (\frac{1-\beta }{2})}{2\sqrt{%
c^{1-\beta }}}.
\end{equation*}

\subsection{Contributions and Competing Interests and contributions}

The author declares that there are no other contributors to this article,
and that he has no competing interests.

\end{document}